\documentclass[a4paper, 12pt]{amsart}

\usepackage{amssymb}
\usepackage{amsthm}
\usepackage{amsmath}

\usepackage{comment}

\usepackage{url}

\usepackage[mathscr]{eucal}

\usepackage{enumitem}

\theoremstyle{plain}
\newtheorem{thm}{Theorem}[section]
\newtheorem{prop}[thm]{Proposition}
\newtheorem{cor}[thm]{Corollary}
\newtheorem{lem}[thm]{Lemma}

\theoremstyle{definition}

\newtheorem{ques}[thm]{Question}

\theoremstyle{remark}
\newtheorem{rmk}{Remark}[section]
\newtheorem*{ac}{Acknowledgements}

\newcommand{\To}{\Longrightarrow}

\newcommand{\nn}{\mathbb{N}}
\newcommand{\zz}{\mathbb{Z}}

\newcommand{\rr}{\mathbb{R}}

\newcommand{\yoemph}[1]{\emph{#1}}

\newcommand{\yoomega}{\omega_{0}}

\newcommand{\yocantorc}{\Gamma}

\newcommand{\yourysp}{\mathbb{U}}
\newcommand{\yourydis}{\rho}

\DeclareMathOperator{\card}{Card}

\DeclareMathOperator{\cl}{CL}

\newcommand{\yosub}{\subseteq}

\DeclareMathOperator{\yodiam}{diam}

\newcommand{\yocmet}[2]{\mathrm{UCPM}(#1, #2)}

\newcommand{\yocpm}[1]{\mathrm{CPM}(#1)}

\newcommand{\met}[1]{\mathrm{Met}(#1)}
\DeclareMathOperator{\metdis}{\mathcal{D}}

\DeclareMathOperator{\umetdis}{\mathcal{UD}}

\newcommand{\yochara}{\mathscr}

\newcommand{\oumetbd}[1]{\mathrm{BMet}(#1)}
\newcommand{\oucpmbd}[1]{\mathrm{BCPM}(#1)}

\newcommand{\oucube}{\mathbf{Q}}

\newcommand{\ouelltwo}{\ell^{2}}
\newcommand{\ouelltwosp}[1]{\ouelltwo(#1)}
\newcommand{\ouellinf}{\ell^{\infty}}
\newcommand{\ouellsp}[1]{\ouellinf(#1)}
\newcommand{\ouelldis}[1]{\lVert #1 \lVert_{\ouellinf}}

\newcommand{\ouonept}{pt}
\newcommand{\ouoneptsp}{\{\ouonept\}}

\newcommand{\yomainmap}{\mathbb{E}}
\newcommand{\yomainmapf}{\mathbb{F}}

\newcommand{\yoextmapb}{\mathbb{B}}
\newcommand{\yoextmap}{\mathbb{E}}

\newcommand{\yoirrsp}[1]{B(#1)}

\newcommand{\oucsep}{\yochara{S}}
\newcommand{\oucwei}[1]{\oucsep_{#1}}
\newcommand{\oucbdd}[1]{\mathrm{B}(#1)}

\newcommand{\oucfuncsp}[1]{C_{b}(#1)}
\newcommand{\oucfuncdis}[1]{\lVert #1 \rVert}
\newcommand{\oucosp}[1]{c_{0}(#1)}
\newcommand{\yooone}{\omega_{1}}

\begin{document}

\title[Isometric universality]{
On isometric universality of   spaces of metrics}

\author{Yoshito Ishiki \and Katsuhisa Koshino}

\address[Yoshito Ishiki]{
Department of Mathematical Sciences
\endgraf
Tokyo Metropolitan University
\endgraf
Minami-osawa Hachioji Tokyo 192-0397
\endgraf
Japan}

\email{ishiki-yoshito@tmu.ac.jp}

\address[Katsuhisa Koshino]{
Faculty of Engineering
\endgraf
Kanagawa University
\endgraf Yokohama, 221-8686
\endgraf Japan}

\email{ft160229no@kanagawa-u.ac.jp}



\date{\today}

\begin{abstract}
A metric space $(M, d)$ is said to be universal for
a class of metric spaces if all  metric spaces in the class  can be isometrically embedded into $(M, d)$.  In this paper,  for a metrizable space   $Z$ possessing  abundant subspaces, we first  prove that the space of bounded metrics on $Z$ is universal for all bounded metric spaces (with restricted cardinality).  Next, in contrast,  we show that if $Z$ is an infinite  discrete space,  then the space of metrics on $Z$  is   universal for  all separable metric spaces. As a corollary of our results,  if $Z$ is non-compact, or uncountable and compact,   then the space of metrics on $Z$ is   universal for all compact  metric spaces.  In addition, if $Z$ is compact and countable, then there exists a compact metric space that can not be isometrically embedded into the space of metrics on $Z$.  
\end{abstract}

\maketitle

\section{Intoroduction}\label{sec:intro}

\subsection{Backgrounds}

For a metrizable space 
$X$, 
let 
$\yocpm{X}$
be 
the set of all continuous 
pseudometrics on $X$.  
The continuity of a
 pseudometric 
 $d$ 
 on 
 a set 
 $X$ 
 means 
that 
$d$ 
is continuous as a map 
from 
$X\times X$ 
to 
$[0, \infty)$. 
We  denote by 
$\met{X}$
the set of 
metrics on  $X$ that generate the 
same topology of $X$. 
For $d, e\in \yocpm{X}$, 
we define a metric $\metdis_{X}$ by 
\[
\metdis_{X}(d, e)=\sup_{x, y\in X}|d(x, y)-e(x, y)|. 
\]
We represent 
the restricted metric 
$\metdis_{X}|_{\met{X}^{2}}$ 
as the original symbol 
$\metdis_{X}$.
We also denote by 
$\oumetbd{X}$
(resp.~$\oucpmbd{X}$)
the set of all bounded metrics in 
$\met{X}$
(resp.~all bounded pseudometrics in 
$\yocpm{X}$). 
Remark that
if $X$ is compact, 
then 
$\met{X}=\oumetbd{X}$.

The first author 
introduced 
the concept 
of 
$\met{X}$
as a moduli space of 
``geometries''
 on 
$X$, 
and clarified 
the denseness 
and Borel hierarchy 
of 
a subset 
$\{\, d\in \met{X}\mid \text{$(X, d)$ satisfies $\mathcal{P}$}\, \}$
in 
$\met{X}$
for
a certain property 
$\mathcal{P}$ on 
metrics spaces. 
For more information, 
see \cite{Ishiki2020int}, 
\cite{Ishiki2021ultra}, 
\cite{Ishiki2021dense}, 
\cite{MR4527953}, 
\cite{Ishiki2023disco},
and 
\cite{Ishiki2023sr}. 
The second author determined 
topological types of 
$\oumetbd{X}$ with respect to 
not only uniform topologies 
but also 
compact-open topologies
(see
 \cite{koshino2022topological}
and 
\cite{MR4586584}). 
He
\cite{koshino2024borel}
 also proved that 
$\oumetbd{X}$
 is completely metrizable 
if and only if 
$X$ is $\sigma$-compact
when $X$ is separable. 
In this paper, 
we will 
investigate 
the isometric 
 universality 
 of 
 $(\met{X}, \metdis_{X})$. 

Let 
$\yochara{C}$
be a class of 
metric spaces. 
In this setting, 
a metric space 
$(X, d)$
is 
said to be 
\yoemph{universal for $\yochara{C}$}
or 
\yoemph{$\yochara{C}$-universal} if 
every 
$(S, m)\in \yochara{C}$
can be isometrically embedded into 
$(X, d)$. 
Let 
$\oucsep$
the class of 
all separable 
metric spaces.

There are 
so many 
 results on
the universality
for some  classes of metric spaces. 
We focus  only 
on 
theorems related to 
the present paper. 
Fr\'{e}chet
established 
that 
the space 
$(\ouellinf, \ouelldis{*})$
of 
bounded functions on 
$\nn$
is universal for 
$\oucsep$, 
which is today called the 
Fr\'{e}chet 
embedding 
theorem
(see 
\cite[Proposition 1.17]{MR3114782}). 
Note that 
$\ouellinf$
 is not separable. 
As long as the authors know, 
it was Urysohn 
who first constructed 
a separable metric space 
universal for
 $\oucsep$, 
which space is called the 
Urysohn universal 
metric space
$(\yourysp, \yourydis)$
(see \cite{Ury1927}, 
\cite{MR2435148} and 
\cite{MR2277969}). 
Banach
and Mazur
proved that 
the  space
$\oucfuncsp{[0, 1]}$
of 
all 
continuous bounded functions on 
$[0, 1]$
is universal for 
$\oucsep$
(see \cite[Theorem 10 in Chapter XI]{MR0880204}
and 
\cite[Theorem 1.4.4]{MR3526021}, 
see also 
\cite{MR1328375} and \cite{MR1707147}). 
For every  subset 
$R$ of $[0, \infty)$
 with 
 $0\in R$,
and 
for
every 
ultrametrizable space
$X$, 
as a non-Archimedean analogue of 
$\yocpm{X}$, 
we can define 
the space 
$(\yocmet{X}{R}, \umetdis_{X})$
of
$R$-valued  ultrametrics equipped with 
the non-Archimedean sup-metric. 
In this paper, 
we omit the details of 
$(\yocmet{X}{R}, \umetdis_{X})$
(see 
\cite{Ishiki2021ultra}, 
\cite{Ishiki2022factor}, 
\cite{Ishiki2023const},
and  
\cite{Ishiki2023Ury}). 
In 
\cite{Ishiki2023const} 
and 
\cite{Ishiki2023Ury}, 
the first author in the present paper 
showed that, 
for every  subset 
$R$ of $[0, \infty)$
 with 
 $0\in R$,
and 
for
every 
 infinite
 compact ultrametrizable space
 $X$, 
the ultrametric space 
$(\yocmet{X}{R}, \umetdis_{X})$
of continuous 
pseudo-ultrametrics 
is isometric to 
the $R$-Urysohn universal 
ultrametric space, 
i.e., 
 a non-Archimedean analogue of 
$(\yourysp, \yourydis)$. 
Thus
$\yocmet{X}{R}$
is universal 
for 
all 
separable 
$R$-valued 
ultrametric spaces. 
Based on this 
non-Archimedean 
theorem, 
we can consider that 
the present  paper 
deals with 
an  Archimedean analogue of 
the results in 
\cite{Ishiki2023const} 
and 
\cite{Ishiki2023Ury}.

\subsection{Main results}
In this paper, 
for 
a metrizable space $Z$ 
possessing  abundant 
subspaces, 
we first  prove that 
the space of bounded 
metrics on $Z$ is 
universal for all bounded metric spaces
(with restricted cardinality). 
Next, in contrast, 
 we show that if $Z$ is an  infinite  discrete space, 
 then the space of metrics on $Z$ 
 is   universal for 
 all separable metric spaces. 
Third,  
 we investigate 
 the spaces of metrics on 
 compact and countable metrizable spaces. 
 
\subsubsection{Spaces possessing  abundant subspaces}
Throughout this paper, 
we often  use the 
set-theoretic notation of 
ordinals and cardinals.
The symbol 
$\yoomega$
stands for the first infinite 
ordinal, 
which is equal to 
$\zz_{\ge 0}$ as a set. 
The discrete  topology of 
$\yoomega$ is the same to 
the order topology on 
$\yoomega$. 
Moreover, 
for two ordinals $\alpha, \beta$, 
the relation 
$\alpha<\beta$ means that 
$\alpha \in \beta$. 
In this paper, 
we will deal with 
both  the discrete topology and 
the order topology on 
a cardinal $\kappa$. 
We will 
clarify which topology is being used
to ensure 
that no confusion can arise.

Let 
$\kappa$
be a 
cardinal. 
We denote by 
$\oucwei{\kappa}$
the class of 
all metric spaces 
of weight 
$\kappa$. 
For example, 
$\oucwei{\aleph_{0}}=\oucsep$. 
For a class 
$\yochara{C}$
of 
metric spaces, 
we denote by 
$\oucbdd{\yochara{C}}$
the all bounded metric spaces 
belonging to 
$\yochara{C}$. 
For example, 
the class 
$\oucbdd{\oucsep}$
consists of 
all separable bounded metric spaces. 

For two topological spaces $X$ and $Y$, 
we denote by 
$X\oplus Y$
the topological sum of $X$ and $Y$. 
Let $\ouoneptsp$
denote 
the one-point space. 

The next theorem is our first result. 
\begin{thm}\label{thm:univ0}
Let 
$Z$
 be a metrizable space, 
 and 
 let 
 $(X, d)$
  be a
  metric space such that:
  \begin{enumerate}[label=\textup{(U\arabic*)}, series=univ]
  \item\label{item:univ0main}
there exists 
a metric  space 
$(Y, D)$  containing  $X$ 
as a subspace 
such that  
 $D|_{X^{2}}=d$, 
and there exists 
a closed
topological  embedding from 
$Y \oplus \ouoneptsp$ 
to 
$Z$.
\end{enumerate}
Then 
there 
exists an isometric embedding 
$J\colon (X, d)\to (\met{Z}, \metdis_{Z})$. 
Moreover, if $d$ is bounded, 
we can choose $J$ so that 
$J$ is a map from $(X, d)$ into 
$\oumetbd{Z}$. 
\end{thm}

As our second result, 
Theorem
 \ref{thm:univ0}
yields 
the following theorem 
stating that
if a metric space 
$Z$
contains abundant spaces as subspaces, 
then 
$\met{Z}$
is universal. 
\begin{thm}\label{thm:universal1}
Let 
$\yochara{C}$ 
be a class of 
 metrizable spaces and 
$Z$
 be a metrizable space. 
 Assume that the following condition
 is satisfied:
 \begin{enumerate}[resume*=univ]
 \item\label{item:univ1main}
for each 
$(X, d) \in \yochara{C}$, 
there exists 
a metric   space 
$(Y, D)$  
containing  $X$ 
as a subspace 
such that 
 $D|_{X^{2}}=d$, 
and there exists 
a closed
topological  embedding from 
$Y \oplus \ouoneptsp$ 
to 
$Z$.
\end{enumerate}
Then 
 the space 
 $(\met{Z}, \metdis_{Z})$
 is universal for 
 $\yochara{C}$, 
 and 
 $(\oumetbd{Z}, \metdis_{Z})$
 is universal for 
 $\oucbdd{\yochara{C}}$. 
\end{thm}

We denote by  $\oucube$ 
 (resp.~$\yocantorc$)
 the Hilbert cube, 
 i.e, the countable power of $[0, 1]$ 
 (resp.~the Cantor set, i.e., the countable power of 
 $\{0, 1\}$). 
 Namely, 
 $\oucube=[0, 1]^{\aleph_{0}}$
 and 
 $\yocantorc=\{0, 1\}^{\aleph_{0}}$. 
 For each 
 infinite cardinal 
 $\kappa$, 
 we also denote 
 by 
 $\ouelltwosp{\kappa}$
 the Hilbert space of weight $\kappa$
 (see \cite[p.16]{MR3099433}). 
 Let 
 $\yoirrsp{\kappa}$
 be the  countable power of 
 the discrete space of cardinality
 $\kappa$. 
 Namely, 
 $\yoirrsp{\kappa}=\kappa^{\aleph_{0}}$, 
 where 
 $\kappa$ is 
 equipped with 
 the discrete topology. 

We denote by 
$\yochara{T}$
the class of 
all totally bounded metric spaces. 
As a corollary   of 
Theorem \ref{thm:universal1}, 
we  provide 
several 
universality of 
spaces of metrics. 

\begin{thm}\label{thm:concrete}
The following statements are true: 
\begin{enumerate}[label=\textup{(\arabic*)}]

 \item\label{item:concelltwo}
For every infinite cardinal
$\kappa$, 
the space 
 $(\met{\ouelltwosp{\kappa}}, \metdis_{\ouelltwosp{\kappa}})$
 and 
 $(\met{\yoirrsp{\kappa}}, \metdis_{\yoirrsp{\kappa}})$
 are universal 
 for 
 $\oucwei{\kappa}$. 
Moreover, 
the spaces 
 $(\oumetbd{\ouelltwosp{\kappa}}, \metdis_{\ouelltwosp{\kappa}})$
 and $(\oumetbd{\yoirrsp{\kappa}}, \metdis_{\yoirrsp{\kappa}})$
 are universal 
 for 
 $\oucbdd{\oucwei{\kappa}}$

 \item\label{item:conccube}
The space 
  $(\oumetbd{\oucube}, \metdis_{\oucube})$
  is universal
   for 
   $\yochara{T}$.

    \item\label{item:concgamma}
    For any metrizable space 
$Z$ 
such that there exists a continuous surjection from 
$Z$ to 
$\oucube$, 
the space 
$(\oumetbd{Z}, \metdis_{Z})$
 is universal for
   $\yochara{T}$.
   In particular, 
   $(\oumetbd{\yocantorc}, \metdis_{\yocantorc})$
  is universal 
  for 
     $\yochara{T}$.
  
 \item\label{item:concpolish}
 
For every 
 uncountable Polish space 
 $X$, 
 the space 
 $(\oumetbd{X}, \metdis_{X})$
 is universal 
 for
  $\yochara{T}$. 
  
\end{enumerate}
\end{thm}

\begin{rmk}
There exists 
a
continuous surjection
$\beta\colon [0, 1]\to \oucube$
(this is the Hahn--Mazurkiewicz theorem, 
see
\cite[Theorem 31.5]{MR2048350}).  
Thus, 
the unit interval 
$[0, 1]$
satisfies the assumption of 
 Statement 
\ref{item:concgamma}
in 
Theorem
\ref{thm:concrete}.
\end{rmk}

Due to the same argument as Theorem~\ref{thm:universal1}, we can construct some isometric embeddings from a
given 
 bounded metric space 
$(X, d)$
 into 
$\met{X}$
  and 
  $\oumetbd{X}$.
Namely, 
the following 
  theorem is  an 
  analogue of 
  the Kuratowski 
  embedding theorem
  for spaces of 
  metrics
  (for this famous embeddings, see 
  \cite[p.100]{MR3363168}).

  \begin{thm}\label{thm:self0}
  Let 
  $X$ 
  be a metrizable space, 
  and 
  take 
  $d\in \met{X}$. 
 Assume that either of the following is true:
 \begin{enumerate}[label=\textup{(\arabic*)}]

 \item\label{item:self0uniform}
  the space $X$ 
  contains 
  a subset $S$ 
  such that 
  $S$ 
   is 
   uniformly homeomorphic to 
  $X$ 
 (with respect to $d$) 
  and its closure 
  $\cl_{X}(S)$ 
  is a proper subset of 
  $X$; 
 \item\label{item:self0closed}
 the space $X$ contains 
 a closed proper subset 
 $S$ 
 homeomorphic to $X$. 
 
 \end{enumerate}
 Then 
 there exits an isometric embedding from 
 $(X, d)$ into 
 the space 
 $(\met{X}, \metdis_{X})$. 
 Moreover, 
 if 
 $(X, d)$ is bounded, 
 it can be isometrically embedded into 
 $(\oumetbd{X}, \metdis_{X})$. 
  \end{thm}

\subsubsection{Discrete spaces}
In the theorems 
explained  above, 
we focus on a space $X$ that 
contains 
 abundant spaces as 
subspaces. 
In the contrast, 
we next shall consider 
 discrete spaces. 
For a set 
$S$, 
we denote by 
$\ouellsp{S}$
the set of all bounded functions 
from 
$S$ 
to 
$\rr$ 
equipped with 
the sup-metric 
$\ouelldis{*}$
($\ouellinf$-norm). 
Based on 
 a different method from 
the above 
theorems, 
we verify 
the universality of 
the space of metrics on 
 discrete spaces.

\begin{thm}\label{thm:univ2k}
Let 
$\kappa$
be an infinite 
 cardinal,  
 and regard it as a discrete space. 
 Then 
the following statements are true:
\begin{enumerate}[label=\textup{(\arabic*)}]
\item\label{item:univ2ellisomk}
The space 
$(\ouellsp{\kappa}, \ouelldis{*})$ 
can be isometrically embedded into 
the space 
$(\met{\kappa}, \metdis_{\kappa})$.
In particular, 
the space 
$(\met{\kappa}, \metdis_{\kappa})$
is universal 
for 
  $\oucwei{\kappa}$. 
\item\label{item:univ2bddisomk}
The space 
$(\oumetbd{\kappa}, \metdis_{\kappa})$
is universal 
for 
  $\oucbdd{\oucwei{\kappa}}$. 
  
\end{enumerate}
\end{thm}

\begin{cor}\label{cor:ukappa}
Let 
$Z$
be a 
metrizable space, 
and 
$\kappa$ be an 
infinite cardinal. 
Assume that
either of the following 
is true:
\begin{enumerate}[label=\textup{(\arabic*)}]
\item\label{item:ukappaa}
the space 
$Z$
contains 
a discrete space 
of cardinality
$\kappa$;
\item\label{item:ukappab}
the space 
$Z$
 has 
 weight 
 $\kappa$, 
 and 
 $\kappa$
 has uncountable 
 cofinality. 
\end{enumerate}
Then 
$(\met{Z}, \metdis_{Z})$
is universal for 
$\oucwei{\kappa}$, 
and 
$(\oumetbd{Z}, \metdis_{Z})$
 is universal 
 for 
 $\oucbdd{\oucwei{\kappa}}$. 
\end{cor}

Concentrating on 
the universality for 
 separable or compact metric spaces, 
 we summarize 
 our  theorems appearing above 
 as follows: 
\begin{cor}\label{cor:univ3}
Let $X$ be an infinite metric space. 
\begin{enumerate}[label=\textup{(\arabic*)}]
\item\label{item:iniv3cptunc} 
If 
$X$
 is 
 compact and uncountable, 
then 
$(\met{X}, \metdis_{X})$ 
is 
universal 
 for 
$\yochara{T}$. 
\item\label{item:iniv3ncpt} 
If $X$ is non-compact, 
then 
$(\met{X}, \metdis_{X})$ 
contains an isometric copy
 of 
$(\ouellinf, \ouelldis{*})$. 
In particular, 
$(\met{X}, \metdis_{X})$ 
is universal for 
$\oucsep$. 
Moreover, the space 
$(\oumetbd{X}, \metdis_{X})$
is universal for 
$\oucbdd{\oucsep}$. 
\end{enumerate}
As a consequence, 
if 
$X$ is non-compact, or compact and uncountable, 
then
$(\oumetbd{X}, \metdis_{X})$ is universal 
for all 
compact metric spaces. 
\end{cor}

\subsection{Case of compact and countable spaces}
Based on 
Corollary 
\ref{cor:univ3}, 
a next point of concern is 
what happens for 
$\met{X}$
when $X$ is compact and countable. 
Employing 
results from  the theory of Banach spaces, 
we prove  the  theorem concerning  the lack of the universality. 
\begin{thm}\label{thm:nonuniv}
Let 
$X$ be a  compact 
and countable 
metrizable space. 
Then the space 
$(\met{X}, \metdis_{X})$
is not  
universal for 
all compact metric spaces. 
Namely, 
there exists a 
compact metric space 
that can not be isometrically embedded into 
$(\met{X}, \metdis_{X})$. 
\end{thm}

We denote by 
$\oucosp{\yoomega}$
the set of 
all 
maps $f\colon \yoomega\to \rr$
such that 
$\lim_{n\to \infty}f(n)=0$, 
equipped with
the supremum metric. 
In contrast to 
Theorem 
\ref{thm:nonuniv}, 
we  obtain the next 
isometric embeddings 
from 
compact subsets of 
$\oucosp{\yoomega}$
into 
$\met{X}$, 
where 
$X$
 is compact and 
countable.

\begin{thm}\label{thm:countuniv}
Let 
$X$ be a 
compact 
and 
countably infinite 
metrizable space. 
Then 
every 
compact subset 
of 
$\oucosp{\yoomega}$
can be 
isometrically embedded
into 
the space 
$(\met{X}, \metdis_{X})$. 
Namely, 
the space 
$(\met{X}, \metdis_{X})$
is universal 
for all 
compact subsets of 
$\oucosp{\yoomega}$. 
\end{thm}

Combining Theorem \ref{thm:countuniv}
and  Aharoni's 
result 
\cite{MR0511661}, 
we have: 

\begin{cor}\label{cor:bilip}
Let 
$X$
be a 
 compact
 countably infinite   metrizable space. 
Then
every compact  metric space can be 
bi-Lipschitz embedded into 
$\met{X}$. 
\end{cor}

The organization of the paper is as follows: 
In Section \ref{sec:pre}, 
we prepare basic statements
on spaces on metrics, 
and extensions of metrics. 
Section \ref{sec:zenhan}
is devoted to 
proving 
Theorems 
\ref{thm:concrete}
and 
\ref{thm:self0}. 
In Section \ref{sec:kouhan}, 
we prove 
Theorem \ref{thm:univ2k}
and Corollaries  
\ref{cor:ukappa} and 
\ref{cor:univ3}. 
Section 
\ref{sec:countable}
shows 
Theorems
\ref{thm:nonuniv}, 
\ref{thm:countuniv}, 
and 
Corollary 
\ref{cor:bilip}. 
Section 
\ref{sec:ques}
exhibits 
several questions 
on the universality of 
spaces of metrics. 

\begin{ac}
The first author 
 was supported by JSPS 
KAKENHI Grant Number 
JP24KJ0182. 

\end{ac}

\section{Preliminaries}\label{sec:pre}
In this section, 
we review basic facts on 
spaces of metrics. 
\subsection{Basic facts on spaces of metrics}
\begin{lem}\label{lem:wa}
Let 
$X$ be a
metrizable 
space. 
Then 
we have 
\[
\yocpm{X}+\met{X}\yosub \met{X}.
\] 
Namely, 
if 
$d\in \yocpm{X}$ and 
$e\in \met{X}$, 
then 
$d+e\in \met{X}$. 
Moreover, 
if both  $d$ and $e$ are bounded, 
the so is 
$d+e$. 
\end{lem}
\begin{proof}
As is easily observed, 
$d + e$
 is a metric on 
 $X$.
Define 
a
 map 
$f\colon(X, d+e)\to (X, e)$
 by 
$f(x)=x$. 
Since 
$d+e$ is continuous on 
$X\times X$,
all open balls of 
$(X, d+e)$ 
are
also 
 open in 
 $X$. 
Thus 
the map 
$f$
 is an open map. 
By 
$e(f(x), f(y))\le (d+e)(x, y)$
for all 
$x, y\in X$, 
the map  
$f$ 
is 
$1$-Lipschitz, 
in particular, it is continuous. 
This implies that
$f$
 is a homeomorphism. 
Hence 
$d+e\in \met{X}$. 
\end{proof}

\begin{cor}\label{cor:wawa}
Let 
$\yochara{C}$
 be a 
 class of metric spaces. 
If
$X$ is a 
metrizable space, 
and  
the space $\yocpm{X}$
(resp.~$\oucpmbd{X}$)
 is
universal for 
$\yochara{C}$, 
then 
so is 
the space 
$\met{X}$
(resp.~$\oumetbd{X}$). 
\end{cor}
\begin{proof}
We will show that 
for every 
$(Z, w)\in \yochara{C}$, 
there exists 
an isometric embedding 
$f\colon Z\to \met{X}$. 
Using the assumption, 
we obtain 
an isometric embedding 
$g\colon (Z, w)\to (\yocpm{X}, \metdis_{X})$. 
Take a metric 
$d$ 
in 
$\met{X}$, 
and define a map 
$f\colon Z\to \yocpm{X}$
by 
$f(z)=g(z)+d$. 
Lemma \ref{lem:wa}
implies that 
$f$
 is
 actually 
  a map into 
$\met{X}$. 
Since  we have 
$\metdis_{X}(f(x), f(y))=
\metdis_{X}(g(x), g(y))$, 
we conclude that 
$f$ is an isometric embedding. 
\end{proof}

\begin{cor}\label{cor:equivmetcpm}
Let 
$\yochara{C}$
 be a 
 class of metric spaces, 
 and 
$X$ 
be a 
metrizable space. 
Then 
the space 
$\yocpm{X}$
(resp.~$\oucpmbd{X}$)
is universal for 
$\yochara{C}$ if 
and only if 
so is $\met{X}$
(resp.~$\oumetbd{X}$). 
\end{cor}

Let $X$ and $Y$
 be a 
 metrizable spaces, 
 and 
 $f\colon X\to Y$
 be a
 continuous map. 
 Then
 for each 
 $d\in \yocpm{Y}$, 
  we define a pseudometric 
  $f^{*}d$ on $X$ by 
 $f^{*}d(x, y)=d(f(x), f(y))$. 
We denote by 
$f^{*}\colon \yocpm{Y}\to \yocpm{X}$
the map defined by 
$d\mapsto f^{*}d$. 
\begin{lem}\label{lem:univsurj}
Let 
$X$ 
and 
$Y$
be 
metrizable spaces, 
and let 
$f\colon X\to Y$ be a continuous map. 
Assume that 
$f$
is surjective. 
Then the following statements are true:
\begin{enumerate}[label=\textup{(\arabic*)}]
\item\label{item:surjstar}
The map 
$f^{*}\colon \yocpm{Y}\to \yocpm{X}$
is an isometric embedding
such that 
$f^{*}(\oucpmbd{Y})\yosub \oucpmbd{X}$.
\item\label{item:surjstarstar}
There exists an isometric embedding 
$h\colon \met{Y}\to \met{X}$
such that 
$h(\oumetbd{Y})\yosub \oumetbd{X}$.
\item\label{item:surjuniv}
For a class 
$\yochara{C}$
of metric spaces, 
if  $\met{Y}$ 
(resp.~$\oumetbd{Y}$)
is universal for 
 $\yochara{C}$, 
then so is 
$\met{X}$
(resp.~$\oumetbd{X}$). 
\end{enumerate}
\end{lem}
\begin{proof}
First let us verify 
\ref{item:surjstar}. 
Since 
$f$ is surjective, 
for every pair 
$d, e\in \yocpm{Y}$
we have 
\begin{align*}
&
\metdis_{X}(f^{*}d, f^{*}e)=\sup_{x, y\in X}|f^{*}d(x, y)-f^{*}e(x, y)|
\\
&=
\sup_{x, y\in X}|d(f(x), f(y))-e(f(x), f(y))|\\
&=
\sup_{u, v\in Y}|d(u, v)-e(u, v)|
=\metdis_{Y}(d, e). 
\end{align*}
This proves 
\ref{item:surjstar}. 
Next, we show 
\ref{item:surjstarstar}. 
Take 
$\rho\in \oumetbd{X}$
and define 
a map $h\colon \met{Y}\to \met{X}$ by 
$h(d)=f^{*}d+\rho$. 
Then the map 
$h$ is as desired. 
Statement 
\ref{item:surjuniv} is deduced from 
Statement 
\ref{item:surjstarstar}. 
\end{proof}

\begin{prop}\label{prop:denseisom}
Let 
$X$ 
be a metrizable space, 
and 
$H$
 be a dense subset of 
 $X$. 
 Then the map 
 $r\colon \met{X}\to \met{H}$
 defined by 
 $r(d)=d|_{H^{2}}$
 is isometric embedding
 and 
 $r(\oumetbd{X})\yosub\oumetbd{H}$. 
\end{prop}
\begin{proof}
For every pair 
$d, e\in \met{X}$, 
define a 
map 
$s\colon X\times X\to [0, \infty)$
 by 
$s(x, y)=|d(x, y)-e(x, y)|$. 
Then it is 
continuous. 
Since $H$ is dense in 
$X$, 
we see that 
$H\times H$ is dense in 
$X\times X$. 
Thus, 
we obtain 
$\sup_{(x, y)\in X\times X}s(x, y)
=\sup_{(u, v)\in H\times H}s(u, v)$. 
This identity is equivalent to 
$\metdis_{X}(d, e)
=\metdis_{H}(r(d), r(e))$. 
This completes the proof. 
\end{proof}

\subsection{Extensions of metrics}
Before proving our main results, 
we will 
 cite  theorems on 
extensions of metrics. 
Combining the efforts by 
C. Bessaga \cite{MR1211761}, 
T. Banakh \cite{MR1811849}, 
O. Pikhurko \cite{MR1692801} 
and 
M. Zarichnyi \cite{zarichnyi1996regular}, we have the following theorem
(see \cite[Theorem 1.2]{MR2520139}). 

\begin{thm}\label{thm:extensionsugoi}
Let
 $(X, d)$ 
be a metric space and 
$A$
 be a closed subset of 
 $X$.
There is a continuous function 
$\yoextmapb \colon  \oumetbd{A} \to
\oumetbd{X}$ 
such that 
$\yoextmapb(d)|_{A^2} = d$ 
and
$\metdis_{A}(d, \rho)
=\metdis_{X}(\yoextmapb(d), \yoextmapb(\rho))$
 for all metrics 
 $d, \rho \in \oumetbd{A}$.
\end{thm}

We will also make use of 
the first author's isometric extension theorem of 
metrics 
\cite[Theorem 1.1]{Ishiki2024isomope}
that treats not only  bounded metrics, but also unbounded ones. 
\begin{thm}\label{thm:main1}
Let 
$X$ 
be 
a metrizable space, 
and 
$A$
 be a closed subset of $X$. 
 Then there exists a map 
 $\yomainmap\colon \met{A}\to \met{X}$
 such that 
 \begin{enumerate}[label=\textup{(\arabic*)}]
 \item\label{item:m1ext}
for every 
$d\in \met{A}$ we have 
$\yomainmap(d)|_{A^{2}}=d$; 
\item\label{item:m1isom}
the map $\yomainmap$  is 
an isometric embedding, i.e., 
for every pair  $d, e\in \met{A}$, 
we have 
\[
\metdis_{A}(d, e)=\metdis_{X}(\yomainmap(d),
\yomainmap(e));
\]
\item\label{item:m1bounded}
$\yomainmap(\oumetbd{A})\yosub \oumetbd{X}$; 

\item\label{item:m1comp}
if $d\in \met{A}$ is complete, 
then so is 
$\yomainmap(d)$. 
 \end{enumerate}
 Moreover, 
there exists an isometric operator 
extending pseudometrics 
$\yomainmapf\colon \yocpm{A}\to \yocpm{X}$
such that 
$\yomainmapf|_{\met{A}}=\yomainmap$. 
\end{thm}

\subsection{Embeddings into linear  spaces}

\begin{lem}\label{lem:gfrechet}
Let 
$\kappa$
 be an infinite cardinal. 
The space 
$(\ouellsp{\kappa}, \ouelldis{*})$ 
is universal for 
 $\oucwei{\kappa}$. 
\end{lem}
\begin{proof}
The proof is the same 
as that of 
the 
Fr\'{e}chet 
embedding 
theorem
(see 
\cite[Proposition 1.17]{MR3114782}). 
For $(X, d)\in \oucwei{\kappa}$, 
take a dense subset 
$\{q_{\alpha}\}_{\alpha<\kappa}$, 
and 
define 
$f\colon X\to\ouellsp{\kappa}$
by 
$x\mapsto (d(x, q_{\alpha})-d(q_{\alpha}, q_{0}))_{\alpha<\kappa}$. 
Then $f$ is an isometric embedding
from $(X, d)$ into $(\ouellsp{\kappa}, \ouelldis{*})$. 
\end{proof}

\begin{lem}\label{lem:finlem00}
Let 
$\kappa$
 be an infinite cardinal. 
For every 
completely 
metrizable  space 
$Z$ of 
weight
 $\kappa$, 
there exists a 
topological 
embedding 
$f\colon Z\to \ouelltwosp{\kappa}$
such that 
$f(Z)$
is closed in 
$\ouelltwosp{\kappa}$. 
\end{lem}
\begin{proof}
See \cite[Theorem 6.2.4]{MR3099433}. 
\end{proof}

\section{Spaces with abundant subspaces}\label{sec:zenhan}

To prove our main results, 
we will show the following lemma:

\begin{lem}\label{isometric}
Let 
$X$ 
be a metrizable space. 
Then 
for every metric 
$d\in \met{X}$,  
there is an isometric embedding form 
$(X, d)$ 
into 
the space 
$\left(\met{X \oplus \ouoneptsp},\metdis_{X \oplus \ouoneptsp}\right)$.
Moreover, 
if
 $d$ is bounded, 
then 
$(X, d)$
 can
be isometrically embedded into 
$\oumetbd{X\oplus \ouoneptsp}$. 
\end{lem}

\begin{proof}
For each 
$u \in X$, 
we shall define 
$d_{u} \in \met{X \oplus \ouoneptsp}$ 
as follows:
 \[
 d_{u}(x, y)=
 \begin{cases}
 d(x, y) & \text{if $x, y\in X$;}\\
 d(x, u)+1 & \text{ if $x\in X$ and $y=\ouonept$;}\\
 d(y, u)+1 & \text{if $x=\ouonept$ and $y\in X$;}\\
 0 & \text{if $x=y=\ouonept$}
 \end{cases}
 \]
Define  
$I\colon X\to \met{X \oplus \ouoneptsp}$
 by 
$u \mapsto d_{u}$. 
Let us verify that 
the map 
$I$
 is an isometric embedding.
Fix any 
$u, v \in X$ 
and any 
$x, y \in X\oplus \ouoneptsp$.
By the definition of 
$d_{u}$ 
and 
$d_{v}$,
 if 
 $(x, y)\in X^{2}$ 
 or 
  $x = y = \ouonept$,
 then 
 $|d_{u}(x, y) - d_{v}(x, y)| = 0$.
In the case where 
$x \in X$ and
 $y = \ouonept$,
 we obtain 
 \[
 |d_{u}(x, y) - d_{v}(x, y)| = 
 |(d(x, u) + 1) - (d(x, v) + 1)| \leq 
 d(u, v).
 \]
Similarly, when 
$x = \ouonept$ 
and 
$y \in X$,
we also obtain 
 $|d_{u}(x, y) - d_{v}(x,y)| \leq d(u, v)$.
Hence 
$\metdis_{X\oplus \ouoneptsp}(d_{u},d_{v}) 
\leq d(u, v)$.
Moreover, we also  have 
\[
|d_{u}(u, \ouonept) - d_{v}(u, \ouonept)| = 
|(d(u, u) + 1) - (d(u, v) + 1)| = d(u, v). 
\]
This  implies that 
 $\metdis_{X\oplus \ouoneptsp}(d_{u},d_{v}) = 
 d(u, v)$.
Therefore, 
the map
$I$ 
is an isometric embedding.
If $d$ is bounded, 
then each 
$I(z)$ belongs to 
$\oumetbd{X\oplus \ouoneptsp}$. 
This finishes the proof. 
\end{proof}

Now we shall prove Theorems \ref{thm:univ0}--\ref{thm:self0}.

\begin{proof}[Proof of Theorem \ref{thm:univ0}]
Take  
an extension  
$(Y, D)$ 
of 
$X$
stated in 
\ref{item:univ0main}. 
It follows from 
Lemma
\ref{isometric}
 that there exists an isometric embedding 
$I \colon (Y, D) \to 
(\met{Y \oplus \ouoneptsp}, \metdis_{Y \oplus \ouoneptsp})$.
By virtue of 
Theorem \ref{thm:main1}, 
we can obtain a 
continuous operator  
$\yoextmap \colon  \met{Y \oplus \ouoneptsp} \to \met{Z}$
such that 
\begin{enumerate}
\item
 for any 
$d \in \met{Y \oplus \ouoneptsp}$, 
we have 
$\yoextmap(d)|_{(Y \oplus \ouoneptsp)^{2}} = d$; 
\item   
for any 
$d, \rho \in \met{Y \oplus \ouoneptsp}$, 
we have 
\[
\metdis_{Y\oplus \ouoneptsp}(d, \rho)
=\metdis_{Z}(\yoextmap(d), \yoextmap(\rho)).
\] 
\end{enumerate}
Write
$\iota\colon (X, d)\to (Y, D)$
as the  inclusion map. 
Thus, 
the composition 
$J=\yoextmap \circ I \circ \iota
\colon  
(X, d) \to (\met{Z}, \metdis_{Z})$
 is an isometric embedding.
This completes the proof of 
Theorem \ref{thm:universal1}. 
\end{proof}

\begin{proof}[Proof of  Theorem \ref{thm:universal1}]
Theorem \ref{thm:universal1}
is a direct consequence of 
Theorem 
\ref{thm:univ0}. 
\end{proof}

\begin{proof}[Proof of Theorem \ref{thm:concrete}]
First, we show the universality of 
$\met{\ouelltwosp{\kappa}}$. 
Take $(X, d)\in \oucwei{\kappa}$. 
Let 
$(Y, D)$
be the completion of 
$(X, d)$. 
Notice that 
the space 
$Y\oplus \ouoneptsp$
is completely metrizable. 
Due to 
Lemma
\ref{lem:finlem00}, 
there exists a topological embedding
 $f\colon Y\oplus \ouoneptsp\to \ouelltwosp{\kappa}$
 such that 
 $f(Y\oplus \ouoneptsp)$ is closed in  
 $\ouelltwosp{\kappa}$. 
Then 
 the condition 
\ref{item:univ1main} in
Theorem  \ref{thm:universal1} 
is fulfilled. 
Thus, 
Theorem 
\ref{thm:universal1} 
implies that 
the space 
$(\met{\ouelltwosp{\kappa}}, \metdis_{\ouelltwosp{\kappa}})$
is universal 
for $\oucwei{\kappa}$
(resp.~ the space $(\oumetbd{\ouelltwosp{\kappa}}, \metdis_{\ouelltwosp{\kappa}})$
is universal for  
$\oucbdd{\oucwei{\kappa}}$). 
We next  consider
$\yoirrsp{\kappa}$. 
Employing 
Engelking's  result
\cite[Theorem]{MR239571}
stating that  that  
every completely  metrizable space
of weight $\kappa$
 is a closed image of 
 $\yoirrsp{\kappa}$, 
 we obtain 
 a closed surjective map
 $h\colon \yoirrsp{\kappa}\to \ouelltwosp{\kappa}$. 
 Then 
Lemma 
\ref{lem:univsurj}
shows that 
$\met{\ouelltwosp{\kappa}}$ 
is 
a metric subspace of 
$\met{\yoirrsp{\kappa}}$
and 
$\oumetbd{\ouelltwosp{\kappa}}$ 
is 
a metric subspace of 
$\oumetbd{\yoirrsp{\kappa}}$. 
Hence 
$(\met{\yoirrsp{\kappa}}, \metdis_{\yoirrsp{\kappa}})$
 is also universal 
 for 
 $\oucwei{\kappa}$
 and 
 $(\oumetbd{\yoirrsp{\kappa}}, \metdis_{\yoirrsp{\kappa}})$
 is  universal 
 for 
 $\oucbdd{\oucwei{\kappa}}$. 
 
 Let us prove 
 Statement 
 \ref{item:conccube}. 
 Similarly to the above proof, 
 take 
 $(X, d)\in   \yochara{T}$. 
Let 
$(Y, D)$
 be the completion of 
 $(X, d)$. 
 Then 
 $(Y, D)$
  is compact. 
 According  to 
 the  method of 
 the proof of 
the Urysohn metrization theorem 
(see, for example,  
\cite[Theorem 23.1]{MR2048350}
and \cite[Theorem 34.1]{MR3728284}), 
there exists a 
topological embedding
$f\colon Y\oplus \ouoneptsp\to \oucube$. 
Since 
$Y$ is compact, 
the embedding  $f$ is  a closed map. 
This means that 
 the condition 
\ref{item:univ1main} in
Theorem  \ref{thm:universal1} 
is fulfilled. 
Using  Theorem 
\ref{thm:universal1} again, 
we see that 
$\met{\oucube}, \metdis_{\oucube})$ is 
universal for 
$\yochara{T}$.

Now we shall show 
Statement \ref{item:concgamma}. 
Take a 
continuous surjection 
$f\colon Z\to \oucube$. 
Thus, 
Lemma 
\ref{lem:univsurj} 
shows that 
there exists an isometric embeddings 
from $\oumetbd{\oucube}$
into 
$\oumetbd{Z}$. 
Thus, 
the space 
$(\oumetbd{Z}, \metdis_{Z})$
is 
universal for 
$\yochara{T}$. 
In particular,
since 
there exists
a
continuous surjection
$\alpha\colon \yocantorc\to \oucube$
(see 
\cite[Theorem 30.7]{MR2048350}), 
we see that 
$(\met{\yocantorc}, \metdis_{\yocantorc})$
 is universal  for 
 $\yochara{T}$.

Next let us show 
Statement 
\ref{item:concpolish}. 
Take an arbitrary uncountable 
Polish space 
$X$. 
Then there exists an topological embedding 
$\iota \colon \yocantorc\to X$
(the Cantor--Bendixson theorem, 
see
\cite[Corollary 6.5]{MR1321597}). 
Thus by
Theorem
\ref{thm:extensionsugoi} 
or 
Theorem \ref{thm:main1}, 
we obtain an isometric embedding 
from $\oumetbd{\yocantorc}$
into 
$\oumetbd{X}$. 
Combining this isometric embedding 
and 
Statement 
\ref{item:concgamma}
with respect to 
$\yocantorc$
in  Theorem \ref{thm:concrete}, 
we conclude that 
$(\oumetbd{X}, \metdis_{X})$
is universal for 
$\yochara{T}$. 
\end{proof}

\begin{proof}[Proof  of Theorem \ref{thm:self0}]
We first assume that 
\ref{item:self0uniform}
 is satisfied. 
 Take  a uniform homeomorphism 
$f\colon (S, d|_{S^{2}})\to (X, d)$. 
Since $f$ and $f^{-1}$ is uniformly continuous, 
there exists an extension space 
$(Y, D)$ of $(X, d)$ and a homeomorphism 
$F\colon Y\to \cl_{X}(S)$
such that
$D|_{X^{2}}=d$
and 
$F|_{X}=f$. 
By the assumption
that 
$\cl_{X}(S)$ is a proper subset of 
$X$, 
we can take 
$p\in X\setminus \cl_{X}(S)$. 
Since 
$Y\oplus \ouoneptsp$ is homeomorphic to 
$\cl_{X}(S)\sqcup\{p\}$, 
by 
Theorems \ref{thm:univ0}, 
there exists an isometric embedding 
$(X, d)\to \met{\cl_{X}(S)\sqcup\{p\}}$. 
Due to 
Theorem  \ref{thm:main1}, 
$ \met{\cl_{X}(S)\sqcup\{p\}}$ 
can be isometrically embedded into 
$\met{X}$. 
As a result, 
we also obtain 
an isometric embedding 
$(X, d)$ into 
$\met{X}$.

To prove the case of 
\ref{item:self0closed}, 
take any $p \in X \setminus S$.
Since $X \bigoplus  \ouoneptsp$ is homeomorphic to $S \cup \{p\}$,
due to 
Theorems \ref{thm:univ0} and  \ref{thm:main1},
there is an isometric embedding from
 $(X,d)$ into
 $\met{X}$. 
This finishes the 
proof. 
\end{proof}

\section{The countable discrete spaces}\label{sec:kouhan}
The purpose of 
this section is 
to prove 
Theorem
\ref{thm:univ2k}
and 
Corollaries 
\ref{cor:ukappa}
and 
\ref{cor:univ3}. 

\begin{lem}\label{lem:www}
Let $X$ be a set. 
If a map 
 $w\colon X\times X \to [0, \infty)$ 
 satisfies:
 \begin{enumerate}[label=\textup{(\arabic*)}]
 \item 
 for 
 every 
 pair 
 $x, y\in X$, 
 we have 
 $d(x, y)=0$ if and only if 
 $x=y$;
 \item 
 for every 
 pair 
 $x, y\in X$, 
 we have 
 $w(x, y)=w(y, x)$; 
 \item\label{cond:www}
 there exists 
 $L\in (0, \infty)$ 
 such that 
 for every 
 distinct pair 
 $x, y\in X$, 
 we have
 $w(x, y)\in [L, 2L]$, 
 \end{enumerate}
 then 
 $w$ 
 is a metric and 
 it generates the discrete topology on 
 $X$. 
\end{lem}
\begin{proof}
It suffices to show that 
$w$ 
satisfies the triangle inequality. 
For every 
triple 
$x, y, z\in X$, 
the condition \ref{cond:www}
implies that 
\[
w(x, y)\le 2L=L+L\le w(x, z)+w(z, y). 
\]
Thus, 
the map 
$w$ satisfies the 
triangle inequality. 
\end{proof}

For a set 
$S$, 
we denote by 
$[S]^{2}$
 the set 
 $\{\, \{p, q\}\mid p, q\in S, p\neq q\, \}$. 
 Namely, 
 $[S]^{2}$
 is the set of all subsets of 
 $S$ 
consisting of exact two points. 
 Note that if $S$ is infinite, 
 then 
 $\card(S)=\card([S]^{2})$, 
 where 
 ``$\card$''
 means the cadinality. 
 Let us provide the 
 proof of 
 Theorem 
 \ref{thm:univ2k}. 
\begin{proof}[Proof of Theorem \ref{thm:univ2k}]
For each 
$t\in [0, \infty)$, 
we define a map 
$R_{t}\colon \rr\to [-t, t]$
by
\[
R_{t}(x)=
\begin{cases}
t & \text{if $t\le x$;}\\
x & \text{if $x\in [-t, t]$;}\\
-t & \text{if $x\le -t$.}
\end{cases}
\]
Note that 
$R_{t}$
 is
  $1$-Lipschitz,
   and 
   $R_{t}$ is the identity on 
   $[-t, t]$. 
   Hence,  
if 
$x, y\in [-t, t]$, 
then we have 
$|R_{t}(x)-R_{t}(y)|=|x-y|$.

Take a family 
$\{S_{i}\}_{i\in \zz_{\ge 0}}$ 
consisting 
of mutually disjoint subspaces of 
$\kappa$
such that 
$\card(S_{i})=\kappa$
 and 
$\kappa=\bigcup_{i\in \zz_{\ge 0}}S_{i}$. 
Using 
$\card([S_{i}]^{2})=\kappa$, 
for each 
$i\in \zz_{\ge 0}$, 
we also take a surjective 
map 
$\tau_{i}\colon [S_{i}]^{2}\to \kappa$. 
For each 
$i\in \zz_{\ge 0}$, 
we define 
a map 
$F_{i}\colon \ouellsp{\kappa}\to \ouellsp{[S_{i}]^{2}}$
by 
\[
F_{i}(a)(\{p, q\})
=R_{2^{i}}\circ a(\tau_{i}(\{p, q\}))+3\cdot 2^{i}. 
\]
Note that for every $i\in \zz_{\ge 0}$, 
 we have 
$F_{i}(a)(\{p, q\}) \in \left[2^{i+1}, 2\cdot 2^{i+1}\right]$
for all 
$a\in \ouellsp{\kappa}$
and 
$\{p, q\}\in [S_{i}]^{2}$. 
We next 
define  a map 
$G_{i}\colon \ouellsp{\kappa}\to \met{S_{i}}$
by 
\[
G_{i}(a)(p, q)=
\begin{cases}
F_{i}(a)(\{p, q\}) & \text{if $p\neq q$;}\\
0 & \text{if $p=q$.}
\end{cases}
\]
Since 
$F_{i}(a)(\{p, q\}) \in \left[2^{i+1}, 2\cdot 2^{i+1}\right]$, 
Lemma
\ref{lem:www} 
($L=2^{i+1}$)
establishes
  that 
the function 
$G_{i}(a)$ 
actually 
belongs to
 $\met{S_{i}}$. 

Now we define  a map 
$H\colon \ouellsp{\kappa}\to \met{\kappa}$ by 
\[
H(a)(p, q)=
\begin{cases}
G_{i}(a)(p, q) & \text{if $p, q\in S_{i}$;}\\
\max\{2^{i+2}, 2^{j+2}\} & \text{if $i\neq j$ and 
$p\in S_{i}$ and $q\in S_{j}$.}
\end{cases}
\]
Note that 
each 
$H(a)$ 
becomes a
metric and 
it
belongs to 
$\met{\kappa}$
since 
the diameter 
$\yodiam_{G_{i}(a)}(S_{i})\le 2^{i+2}$
for all $a\in \ouellsp{\kappa}$ and 
$i\in \zz_{\ge 0}$. 
Now let us verify that 
$H$ is an isometric embedding, 
i.e., 
we will show 
$\metdis_{\kappa}(H(a), H(b))
=\ouelldis{ a-b }$ 
for 
 arbitrary maps
$a, b\in  \ouellsp{\kappa}$. 
Since 
$H(a)(p, q)=H(b)(p, q)$ 
for all $p, q$ such that $p\in S_{i}$
and $q\in S_{j}$ for distinct $i$ and $j$, 
we have 
\begin{align*}
&\metdis_{\kappa}(H(a), H(b))
=\sup_{i\in \zz_{\ge 0}}
\sup_{p, q\in S_{i}}
|H(a)(p, q)-H(b)(p, q)|\\
&=
\sup_{i\in \zz_{\ge 0}}
\sup_{p, q\in S_{i}}
|F_{i}(a)(p, q)-F_{i}(b)(p, q)|\\
&=
\sup_{i\in \zz_{\ge 0}}
\sup_{p, q\in S_{i}}
|R_{2^{i}}(a(\tau_{i}(\{p, q\})))-R_{2^{i}}(b(\tau_{i}(\{p, q\})))|. 
\end{align*}
This means that 
\begin{align*}
&\metdis_{\kappa}(H(a), H(b))
=\tag{A1}\label{al:tag:aaa} \\
&\sup_{i\in \zz_{\ge 0}}\sup_{p, q\in S_{i}}
|R_{2^{i}}(a(\tau_{i}(\{p, q\})))-R_{2^{i}}(b(\tau_{i}(\{p, q\})))|.
\end{align*}
For every $i\in \zz_{\ge 0}$, 
from the fact that 
$R_{2^{i}}$ is 
$1$-Lipschitz, 
it follows that 
\begin{align*}
&|R_{2^{i}}(a(\tau_{i}(\{p, q\})))-R_{2^{i}}(b(\tau_{i}(\{p, q\})))|\\
&\le |a(\tau_{i}(\{p, q\}))-b(\tau_{i}(\{p, q\}))|
\le \ouelldis{ a-b }
\end{align*}
for every $\{p, q\}\in [S_{i}]^{2}$. 
According to the inequality \eqref{al:tag:aaa}, 
we conclude that 
\begin{equation*}
\metdis_{\kappa}(H(a), H(b))\le 
\ouelldis{ a-b }. 
\end{equation*}

Next, we will show the opposite inequality. 
Since $a$ and $b$ are bounded functions, 
we can take a sufficiently large number 
$N\in \zz_{\ge 0}$
such that 
$\{a(s),  b(s)\}\yosub [-2^{N}, 2^{N}]$
for all
 $s\in \kappa$. 
Notice that $N$ is depending on $a$ and $b$. 
Then, since $R_{2^{N}}$ is the identity on 
$[-2^{N}, 2^{N}]$, 
for every $\{p, q\}\in[S_{N}]^{2}$, 
due to 
\eqref{al:tag:aaa}, 
we have 
\begin{align*}
&|a(\tau_{N}(\{p, q\}))-b(\tau_{N}(\{p, q\}))|
\\
&=
|R_{2^{N}}(a(\tau_{N}(\{p, q\})))-R_{2^{N}}(b(\tau_{N}(\{p, q\})))|
\le 
\metdis_{\kappa}(H(a), H(b)).
\end{align*}
Since 
$\tau_{N}$
 is surjection, 
we observe that 
\begin{align*}
\ouelldis{ a-b }\le \metdis_{\kappa}(H(a), H(b)). 
\end{align*}
Then  we conclude that 
$\metdis_{\kappa}(H(a), H(b))
=\ouelldis{ a-b}$. 
Thus, the space 
$\ouellsp{\kappa}$
 is a metric subspace of 
 $\met{\kappa}$. 
 Therefore 
Lemma \ref{lem:gfrechet} completes the proof of 
 Statement 
 \ref{item:univ2ellisomk} in 
the theorem. 

Next, we will show that 
$\oumetbd{\kappa}$ is 
universal for 
all bounded
 metric spaces
 of weight 
 $\kappa$. 
Take an arbitrary bounded  metric space
$(X, d)$ of 
weight 
$\kappa$. 
We can regard 
$(X, d)$ as a
bounded  subspace of 
$\ouellsp{\kappa}$
 by 
Lemma \ref{lem:gfrechet}. 
In this setting, 
we can take a sufficiently large 
$N\in \zz_{\ge 0}$ such that 
$f(s)\in  [-2^{N}, 2^{N}]$
for all
$s\in \kappa$ 
and 
$f\in X$
because 
$X$
 is bounded. 
Using the same method explained above, 
we can isometrically embed 
$(X, d)$
 into 
 $\oumetbd{\coprod_{i=0}^{N}S_{i}}$. 
Since 
$\coprod_{i=0}^{N}S_{i}$
 is homeomorphic to 
 the discrete space
$\kappa$, 
we obtain an isometric embedding 
from 
$(X, d)$ 
into 
$\oumetbd{\kappa}$. 
This proves Statement 
\ref{item:univ2bddisomk} 
in 
the theorem. 
Therefore, 
we completes the 
proof of 
Theorem \ref{thm:univ2k}. 
\end{proof}

\begin{rmk}
We can  obtain another 
proof of 
the universality of 
$\met{\kappa}$ for 
  $\oucwei{\kappa}$
  (Statement 
\ref{item:univ2bddisomk}
and the latter part
of Statement 
 \ref{item:univ2ellisomk}
in Theorem 
\ref{thm:univ2k})
as follows:
Take a dense subset
 $H$ 
 of 
 $\ouelltwosp{\kappa}$
 such that 
$\card(H)= \kappa$,
and take a 
surjection 
$f\colon \kappa\to H$.
Since 
$\kappa$ is equipped with  
the  discrete topology in this setting, 
the map 
$f$ is continuous. 
Then, using 
Lemma 
\ref{lem:univsurj}
and 
Proposition 
\ref{prop:denseisom}, 
 we obtain isometric  embeddings 
\[
\met{\ouelltwosp{\kappa}}\hookrightarrow
\met{H}\hookrightarrow
\met{\kappa}. 
\]
From these isometric embeddings
and 
Statement \ref{item:concelltwo} in Theorem \ref{thm:concrete}, 
 we deduce  that 
$\met{\kappa}$ is 
universal
for $\oucwei{\kappa}$. 
\end{rmk}

\begin{proof}[Proof of Corollary \ref{cor:ukappa}]
Statement 
\ref{item:ukappaa}
follows from 
Theorems \ref{thm:main1}
and 
\ref{thm:univ2k}. 
Next we prove 
\ref{item:ukappab}. 
Since 
$Z$ is 
a metrizable space, 
we can take a family $\{S_{i}\}_{i\in\zz_{\ge 0}}$
of closed discrete  subsets of $Z$ such that 
$\bigcup_{i\in \zz_{\ge 0}}S_{i}$
is dense in $Z$. 
Since 
$Z$ has weight $\kappa$, 
we have 
$\kappa=\card(\bigcup_{i\in \zz_{\ge 0}}S_{i})$. 
Due to the assumption that 
$\kappa$ has uncountable cofinality, 
we can find 
$i\in \zz_{\ge 0}$
such that 
$\card(S_{i})=\kappa$. 
Therefore 
Statement 
\ref{item:ukappab}
is deduced from 
\ref{item:ukappaa}. 
\end{proof}

Recall that 
the symbol 
$\yoomega$ 
stands for the 
countably infinite  discrete space. 
\begin{proof}[Proof of Corollary \ref{cor:univ3}]
First, 
Statement 
\ref{item:iniv3cptunc}
follows from 
Statement 
\ref{item:concpolish}
in 
Theorem \ref{thm:concrete}
since compact metrics spaces are Polish. 
Let us verify
 Statement 
\ref{item:iniv3ncpt}. 
Since 
$X$ is non-compact, 
then 
$X$ 
contains 
a 
closed 
subset 
$S$
that is homeomorphic to 
$\yoomega$. 
Then 
Theorem 
\ref{thm:main1} 
implies that 
there exists an
isometric embedding 
$\yomainmap\colon \met{\yoomega}\to \met{X}$
such that 
$\yomainmap(\oumetbd{\yoomega})\yosub\oumetbd{X}$. 
Thus, Theorem 
\ref{thm:univ2k}
for 
$\yoomega$
proves  Statement 
\ref{item:iniv3ncpt}. 
This finishes the proof of 
Corollary 
\ref{cor:univ3}. 
\end{proof}

\section{Case of  compact countable  spaces}\label{sec:countable}
In this section, 
we will observe 
phenomena 
in 
$\met{X}$
when
$X$
is compact and countable. 

An ordinal 
$\alpha$
is said to
 be 
 a
\emph{succesor}
if  there exists 
an ordinal 
$\beta<\alpha$
such that 
$\alpha=\beta+1$. 
In what follows, 
the symbol 
$\yooone$
stands for  the 
first uncountable 
ordinal. 
We begin with 
the 
classification of 
countable compact spaces. 

\begin{lem}\label{lem:cthomeo}
If 
$X$ 
is
a compact 
metrizable 
space such that 
$\card(X)\le \aleph_{0}$, 
then 
$X$
is homeomorphic to 
a countable successor 
ordinal $\alpha< \yooone$
equipped with the  order topology. 
\end{lem}
\begin{proof}
See \cite{mazurkiewicz1920contribution}, and \cite[Theorem 4]{milliet2011remark}. 
\end{proof}

For the proof of 
Theorem \ref{thm:nonuniv}, 
we prepare several  statements on 
isometric embeddings between 
Banach spaces. 
For a 
topological space 
$X$, 
we denote by 
$\oucfuncsp{X}$
the set of all 
real-valued 
continuous
bounded 
 functions on 
$X$. 
We also denote by 
$\oucfuncdis{*}$
the sup-metric on 
$\oucfuncsp{X}$. 

\begin{thm}\label{thm:noemb}
For a countable successor ordinal 
$\alpha<\yooone$, 
there exists 
a sufficiently large 
countable successor ordinal 
$\theta<\yooone$
such that 
there is no 
isometric embedding 
from $\oucfuncsp{\theta}$
into 
$\oucfuncsp{\alpha}$, 
where
 we regard 
$\theta$ and 
$\alpha$ as ordered topological spaces. 
\end{thm}
\begin{proof}
See 
\cite[Lemma 5]{MR1609675}. 
\end{proof}
\begin{rmk}
For a
topological space
$T$, 
and a countable 
ordinal 
$\alpha<\yooone$, 
we denote by 
$T^{(\alpha)}$
the $\alpha$-th 
Cantor--Bendixson 
derivative of 
$T$. 
The paper  \cite[Lemma 5]{MR1609675}
indicates that, 
in general, 
if a  countable compact metrizable space 
$X$
and a countable ordinal 
$\alpha$ satisfy 
that 
$X^{(\alpha)}=\emptyset$, 
and if 
a compact metrizable space 
 $Y$ 
 satisfies 
 that 
 $Y^{(\alpha+1)}\neq\emptyset$, 
 then 
 the space 
 $\oucfuncsp{Y}$
 can not be isometrically 
 embedded into 
 $\oucfuncsp{X}$. 
\end{rmk}

\begin{thm}\label{thm:soulofsp}
Let 
$R$
be a compact metrizable  space 
and
fix a metric 
$w\in \met{R}$
with 
the diameter 
$\yodiam_{w}(R)\le 1$, 
and let 
$K$ denote  the 
set of all 
$f\in \oucfuncsp{R}$ such that 
$\oucfuncdis{f}\le 1$
and 
$f$ is $1$-Lipschitz 
with respect to $w$.
Then $K$ is 
a
compact subset of 
$\oucfuncsp{R}$, 
and
for every Banach space 
$B$, 
 the following conditions are 
equivalent to 
each other. 
\begin{enumerate}[label=\textup{(\arabic*)}]
\item\label{item:kisom}
 the space
  $B$
contains 
an isometric copy of 
$K$
 equipped with the restricted metric induced by 
the supremum 
norm on 
$\oucfuncsp{R}$;
\item\label{item:linisom}
there exists a 
linear map 
$T\colon \oucfuncsp{R}\to B$
such that 
$T$ is an isometric embedding; 
\item\label{item:isom}
the space 
$\oucfuncsp{R}$
can be  isometrically embedded into 
$B$. 
\end{enumerate}
\end{thm}
\begin{proof}
Using the Arzera--Ascoli Theorem, 
we can observe that 
$K$  
is compact. 
The equivalence between 
\ref{item:kisom} and \ref{item:linisom} follows from 
\cite[Theorem 3.1]{MR2419968}. 
The implication 
``$\ref{item:isom}\To \ref{item:linisom}$''
can be deduced from 
\cite[Corollary 3.3]{MR2030906}
and the fact that 
$\oucfuncsp{R}$
 is separable by the compactness of 
 $R$. 
\end{proof}

We shall  show our remaining main results in this paper.
\begin{proof}[Proof of Theorem \ref{thm:nonuniv}]
Let 
$X$
be a compact 
and countable 
metrizable space. 
Since 
$X^{2}$ 
is
compact
and countable, 
using 
Lemma 
\ref{lem:cthomeo}, 
we can find a 
(countable) successor 
ordinal $\alpha$
homeomorphic to 
$X^{2}$
with respect to 
the order topology on 
$\alpha$. 
By 
Theorem 
\ref{thm:noemb}, 
we can
also  find a sufficiently large 
countable 
successor 
ordinal 
$\theta$
equipped with 
the order topology
 such that 
there does not exists 
an isometric embedding 
from 
$\oucfuncsp{\theta}$
into 
$\oucfuncsp{X^{2}}
(=\oucfuncsp{\alpha})$. 
Take a 
compact subset 
$K$
of 
$\oucfuncsp{\theta}$
mentioned in 
Theorem 
\ref{thm:soulofsp}. 
Then 
Theorem 
\ref{thm:soulofsp} implies that 
the space 
$K$
 can not be isometrically embedded into 
$\oucfuncsp{X^{2}}$. 
Since 
$\met{X}$
 is a metric subspace of 
$\oucfuncsp{X^{2}}$, 
there is no 
isometric embedding 
$K\to \met{X}$. 
Therefore, we 
have constructed a 
compact 
metric 
space that 
can not be isometrically embedded into 
$\met{X}$. 
This finishes the 
proof of 
Theorem 
\ref{thm:nonuniv}. 
\end{proof}

\begin{proof}[Proof of Theorem \ref{thm:countuniv}]
The proof is similar to 
that of 
Theorem \ref{thm:univ2k}. 
Take a compact subset 
$K$
of 
$\oucosp{\yoomega}$. 
Under the assumption, 
$X$ 
is
a compact  and countably infinite 
metrizable space, 
and 
hence
$X$
contains an 
isometric copy of 
$\yoomega+1$. 
Thus
Theorem
\ref{thm:extensionsugoi}
or 
 \ref{thm:main1}
implies that 
$\met{X}$
contains an 
isometric copy 
of 
$\met{\yoomega+1}$. 
It remains to  verify that 
$K$ can be isometrically embedded into 
$\met{\yoomega+1}$, 
where we regard 
$\yoomega+1$ as an ordered topological space, 
thus it is homeomorphic to 
the one-point compactification of 
the 
countably infinite  discrete space. 
We define 
the function
$L\colon  \yoomega\to \rr$
 by 
$L(s)=\inf_{f\in K}f(s)$. 
Since 
$K$ 
is compact, 
for each 
$\epsilon\in (0, \infty)$, 
we can take a finite 
$\epsilon$-net 
$T_{\epsilon}$
of 
$K$. 
We also define 
$h_{\epsilon}\colon \yoomega\to \rr$
by 
$h_{\epsilon}(s)=\min_{f\in T_{\epsilon}}f(s)$. 
Since 
$T_{\epsilon}$
is
a finite
$\epsilon$-net, 
the map  
$h_{\epsilon}$
belongs to 
$\oucosp{\yoomega}$
and 
$\lVert h_{\epsilon}-L\rVert\le \epsilon$. 
Then 
we conclude that 
$L\in \oucosp{\yoomega}$. 
In this setting, 
the map 
$l\in \oucosp{\yoomega}$
defined by 
$l(s)=|L(s)|+2^{-s}$ 
also belongs to 
$\oucosp{\yoomega}$. 
We may assume that 
 every 
$f\in K$ satisfies that 
$f(s)>0$
for all 
$s\in \yoomega$
by replacing 
$K$ with 
$K+l=\{\, f+l\mid f\in K\, \}$
if necessary. 
Put 
$U(s)=\sup_{f\in K}f(s)$. 
Since 
$K$ is 
compact, 
we see that 
$U\in \oucosp{\yoomega}$
and 
$U(s)>0$
by the similar method as $L$. 
Put $S_{k}=\{2k, 2k+1\}$. 
We then define 
$G\colon K\to \met{\yoomega+1}$
by 
\[
G(f)(p, q)=
\begin{cases}
0 & \text{if $p=q$;}\\
f(k) & \text{if $p\neq q$ and $p, q\in S_{k}$;}\\
\max\{U(k), U(l)\} & \text{if $p\in S_{k}$, $q\in S_{l}$,and $k\neq l$;}\\
U(k)& \text{if $p=\yoomega$, and $q\in S_{k}$;}\\
U(k)& \text{if $p\in S_{k}$, and $q=\yoomega$.}
\end{cases}
\]
Under this definition, 
we can see that  
each 
$G(f)$  actually 
belongs to 
$\met{\yoomega+1}$
(compare with \cite[Definition 3.3]{Ishiki2019}). 
Then, 
by the same method to 
the proof of 
Theorem \ref{thm:univ2k}, 
we can observe that 
$G$ is 
an isometric embedding 
from 
$K$
into 
$\met{\yoomega+1}$. 
As a consequence, 
the space
$\met{X}$
also contains 
an isometric copy of 
$K$. 
This finishes the proof. 
\end{proof}

\begin{proof}[Proof of Corollary  \ref{cor:bilip}]
Corollary  \ref{cor:bilip}
follows from 
\cite[Theorem]{MR0511661}. 
\end{proof}

\begin{rmk}
Using 
Theorems 
\ref{thm:soulofsp}
and 
\ref{thm:nonuniv}, 
we see that 
the space of all 
$1$-Lipschitz functions 
$f\colon [0, 1]\to [-1, 1]$
can not be isometrically embedded into 
$\oucfuncsp{\yoomega+1}$
because 
the space
$\oucfuncsp{[0, 1]}$
 is universal for 
 all separable metric spaces
(the Banach--Mazur theorem). 
\end{rmk}

\section{Questions}\label{sec:ques}

As stated in 
Theorem \ref{thm:self0}, 
if 
$X$ 
has abundant subset, 
then 
$\met{X}$
contains metric spaces 
homeomorphic to 
$X$. 
It is interesting 
what happens 
in the case where 
$X$ is finite. 
\begin{ques}
For every 
finite metric space 
$(X, d)$, 
does  there exist
 an 
isometric embedding 
from 
$(X, d)$
into 
$\met{X}$?
\end{ques}

Our results 
states that 
the universality of 
$\oumetbd{X}$ for 
only 
bounded metric spaces. 
Does it have the universality for unbounded spaces?
\begin{ques}
Let 
$X$
be an infinite 
 compact metrizable space. 
 Then, 
is 
$(\oumetbd{X}, \metdis_{X})$
universality for unbounded metric spaces?
\end{ques}

The next question 
asks 
whether $\met{X}$
contains 
an isometric  copy of the Urysohn universal spaces or not. 
\begin{ques}
Let 
$X$ 
be a metrizable space. 
Does  
$(\met{X}, \metdis_{X})$  satisfy the 
finite injectivity?
Namely, 
does  
$(\met{X}, \metdis_{X})$  satisfy 
that 
for every finite metric space 
$(A, m)$, 
 for every subset 
$B$ of $A$, 
and for every isometric embedding 
$f\colon B\to \met{X}$, 
there exists an isometric embedding 
$F\colon A\to \met{X}$ such that 
$F|_{B}=f$?
\end{ques}




\bibliographystyle{myplaindoidoi}
\bibliography{../../../bibtexmet/bibmet.bib}



\end{document}